\documentclass[sn-mathphys, Numbered]{sn-jnl}
\usepackage{graphicx}
\usepackage{multirow}
\usepackage{amsmath,amssymb,amsfonts}
\usepackage{amsthm}
\usepackage{mathrsfs}
\usepackage[title]{appendix}
\usepackage{xcolor}
\usepackage{textcomp}
\usepackage{manyfoot}
\usepackage{booktabs}
\usepackage{algorithm}
\usepackage{algorithmicx}
\usepackage{algpseudocode}
\usepackage{listings}
\usepackage[utf8]{inputenc}
\usepackage[english]{babel}
\usepackage{enumerate}
\usepackage{color}
\usepackage{stmaryrd}
\usepackage{marginnote}
\usepackage{framed}
\usepackage{siunitx}
\usepackage{esint}
\usepackage{mathtools}
\usepackage{microtype}

\usepackage{bm}
\usepackage{scalerel}
\usepackage{tikz}
\usepackage[all]{xy} \xyoption{arc} \xyoption{color}
\usepackage{epsfig}
\usepackage{amsbsy}
\usepackage[font={footnotesize}]{caption}

\usepackage{caption}
\usepackage[skip=0cm,list=true,labelfont=it]{subcaption}
\usepackage{enumitem}

\usepackage{comment}

\newcommand{\eps}{\varepsilon}

\newtheorem{proposition}{Proposition}
\newtheorem{theorem}[proposition]{Theorem}

\newtheorem{remark}[proposition]{Remark}

\newtheorem{definition}[proposition]{Definition}

\numberwithin{equation}{section}
\numberwithin{proposition}{section}

\begin{document}

\title[Stability of structures in a 2D laminar flow]{A measure for the stability of structures immersed in a 2D laminar flow}

\author*[1]{\fnm{Edoardo} \sur{Bocchi}}\email{edoardo.bocchi@polimi.it}

\author*[1]{\fnm{Filippo} \sur{Gazzola}}\email{filippo.gazzola@polimi.it}

\affil[1]{\orgdiv{Dipartimento di Matematica}, \orgname{Politecnico di Milano}, \orgaddress{\street{Piazza Leonardo da Vinci 32}, \city{Milano} \postcode{20133}, \country{Italy}} - MUR Excellence Department 2023-2027}

\abstract{We introduce a new measure for the stability of structures, such as the cross-section of the deck of a suspension bridge,
subject to a 2D fluid force, such as the lift exerted by a laminar wind. We consider a wide class of possible flows, as well as a wide class
of structural shapes. Within a suitable topological framework, we prove the existence of an optimal shape maximizing the stability.
Applications to engineering problems are also discussed.}

\pacs[MSC Classification]{35Q35, 76D05, 74F10}

\maketitle

\section{Introduction}\label{intro}

Let $L>H>0$ and consider the rectangle $R=(-L,L)\times (-H,H)$. Let $B\subset R$ be a compact domain having barycenter at the origin $(x_1, x_2) = (0,0)$ and such that $\ diam(B) \ll L, H$. We study the behavior of a stationary laminar (horizontal) fluid flow crossing $R$ and filling the domain $\Omega= R\setminus B$ with possibly asymmetric inflow and outflow. The fluid is governed by the steady 2D Navier-Stokes equations with inhomogeneous Dirichlet boundary conditions on $\partial \Omega=\partial B \cup \partial R$, see Figure \ref{OmegaB}.

\begin{figure}[h!]
\includegraphics[scale=0.65]{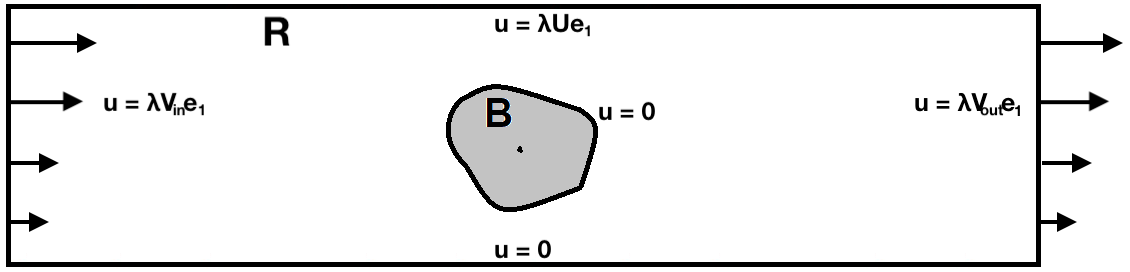}
\vspace{1em}
\caption{The fluid domain $\Omega=R\setminus B$ and the laminar inflow-outflow.}\label{OmegaB}
\end{figure}

A parameter $\lambda\ge0$, proportional to the Reynolds number, measures the strength of the inflow-outflow.
For $U\in \{0,1\}$, we consider the boundary-value problem
\begin{equation}\label{bv-pb}
\begin{aligned}
&- \Delta u + u\cdot \nabla u + \nabla p =0, \quad \nabla\cdot u=0 \quad \mbox{in} \quad \Omega,\\
{u}_{|_{\partial B}}={u}_{|_{\Gamma_b}}&=0, \quad {u}_{|_{\Gamma_t}}=\lambda U e_1, \quad {u}_{|_{\Gamma_l}}=\lambda V_{\rm in}e_1,\quad {u}_{|_{\Gamma_r}}=\lambda V_{\rm out}e_1,
\end{aligned}
\end{equation}
where $\Gamma_b,\, \Gamma_t,\, \Gamma_l,\, \Gamma_r$ denote the sides of $R$ (respectively, bottom, top, left, right) and the couple
$(V_{\rm in}, V_{\rm out})$ belongs to a suitable class of compatible inflow-outflow shapes, see Definition \ref{FU}
in Section \ref{sec-topology}. Throughout the paper we refer to $\lambda(V_{\rm in}, V_{\rm out})$ as the {\em flow}, to $(V_{\rm in}, V_{\rm out})$ as
the {\em flow shape}, and to $\lambda$ as the {\em flow magnitude}.\par
It is known \cite{PaPriDeLan10} that if the solution $(u,p)$ to \eqref{bv-pb} is regular enough, which is the case when $B$ is of class $C^{1,1}$, the (transversal) lift force exerted by the fluid on the body $B$ may be computed through the formula
\begin{equation}\label{lift-strong}
\mathcal{L}_{B}(u,p)=-e_2 \cdot \int_{\partial B} \mathbb{T}(u,p )n,
\end{equation}
where $\mathbb{T}(u,p )= \nabla u + \nabla u^T - p\mathbb{I}$ is the fluid stress tensor and $n$ is the unit outward normal vector to $\partial \Omega$ which, on $\partial B$, points towards the interior of $B$. If the solution $(u,p)$ is not regular, for instance when $B$ is merely Lipschitz, the integral in \eqref{lift-strong} is not defined but one can still compute the lift in a weak sense (see \cite{GazSpe20,BocGaz23}) by
\begin{equation}\label{lift-weak}
\mathcal{L}_B(u,p)= -e_2 \cdot \langle \mathbb{T}(u,p)n,1 \rangle_{\partial B},
\end{equation}where $\langle \cdot, \cdot \rangle_{\partial B}$ denotes the duality pairing between $W^{-2/3,3/2}(\partial B)$ and $W^{2/3,3}(\partial B)$.\par
An apparent paradox in fluid mechanics shows that a perfectly symmetric laminar flow hitting a perfectly symmetric body may generate
a transversal lift force, orthogonal to the flow direction. This paradox is only apparent because it is the {\em asymmetric
	vortex shedding arising leeward} which generates the transversal force. In \cite{GazSpe20} an ``anti-paradox'' was shown, namely that there is no
lift force if the flow has sufficiently small Reynolds number, with an explicit (but far from being sharp) upper bound: the Reynolds number needs
to be so small that it may also be undetectable numerically. Subsequently,
this anti-paradox was used to study equilibrium configurations for fluid-structure interactions (FSI) problems in perfectly symmetric frameworks
\cite{BonGalGaz20,gazzolapatriarca,ClaraNoDEA,BerBonGalGazPer}.
Recently, in \cite{BocGaz23} we considered asymmetric frameworks and we determined a threshold for the Reynolds
number under which the equilibrium configuration is unique, although not symmetric. What is missing to give up invoking the idea of paradox is to establish whether an asymmetric body under the action of an asymmetric flow can maintain its barycenter at the original symmetric position in $x_2=0$.\par
The present paper should be seen as a connection between theoretical results and applications. A first approach in the study of stability of suspension bridges, or other structures interacting with air flows, consists in wind tunnel experiments where artificial laminar flows
simulating winds are created and interact with a scaled model of a symmetric structure. Typically, the flows vanish at the top $\Gamma_t$ and
the bottom $\Gamma_b$ of the inlet-outlet, as for Poiseuille-type flows. Mathematically, this translates into the boundary-value problem
\eqref{bv-pb} with $U=0$. However, this artificial ``best-case'' scenario does not occur in nature because neither winds nor the shape of
bridges are perfectly symmetric and one is led to investigate more general configurations. In fact, for winds hitting real bridges a
Couette-type flow appears more appropriate, which mathematically translates into $U=1$ in \eqref{bv-pb}.\par
The paper is divided in three parts. In Section \ref{sec-topology} we introduce the topological tools needed to rigorously tackle the other two parts.
We define the classes of admissible bodies and flow shapes and we show their compactness in a suitable sense. Then, we derive continuous dependence results for the solution to \eqref{bv-pb} and the associated lift \eqref{lift-weak} with respect to variations of both the
body shape $B$ and the flow $\lambda (V_{\rm in}, V_{\rm out})$.\par
If the body has a symmetric shape (with respect to the horizontal line $x_2=0$)  and the flow has an even shape, there is no lift exerted by the fluid when the flow magnitude $\lambda$ is small, see \cite{GazSpe20}. But if the flow shape is asymmetric, a nonzero lift may appear also on symmetric bodies \cite{BocGaz23}.
In Section \ref{sec-windtunnel} we prove a punctual zero-lift result, which has a direct application in wind tunnel experiments. For a given flow magnitude,
we construct asymmetric bodies and flow shapes for which the lift \eqref{lift-weak} vanishes.\par
The impossibility to prevent the appearance of the lift on a prescribed body shape independent of the flow
$\lambda(V_{\rm in}, V_{\rm out})$, suggests to characterize the shapes of the body $B$ that are more stable.
In Section \ref{sec-measure} we define a new measure for the instability of a structure immersed in a planar fluid, we introduce a quantity capturing the instability of a structure in terms of the lift acting on it in a given range of flow magnitudes $\lambda$.
Stability is meant as minimization of the absolute value of the lift \eqref{lift-weak} considering a suitable class of flows $\lambda(V_{\rm in}, V_{\rm out})$ with a prescribed range of magnitudes. We prove the existence of an optimal body shape,
minimizing the instability in the class of admissible body shapes. The application to bridge design is straightforward: the optimal shape
corresponds to the most stable cross-section for the deck of a suspension bridge.

\section{Preliminary topological setting}\label{sec-topology}

\subsection{Admissible flow and body shapes}
We define the classes of flow and body shapes that will be used throughout the paper.\vspace{1em}
\begin{definition}\label{FU}
Let $U\in\{0,1\}$. We say that $(V_{\rm in},V_{\rm out})$ is an admissible flow shape with controlled norm $r>0$ if\vspace{1em}
\begin{itemize}
\item[(1)] $(V_{\rm in}, V_{\rm out})\in W^{1,\infty}(-H,H)^2$ with $\|V_{\rm in}\|_{W^{1,\infty}(-H,H)}+
\|V_{\rm out}\|_{W^{1,\infty}(-H,H)}\leq r$, \par
\item[(2)] $V_{\rm in}(-H)= V_{\rm out}(-H)=0$ and $V_{\rm in}(H)=V_{\rm out}(H)=U$,\par
\item[(3)] $\int_{-H}^{H}V_{\rm in}(x_2)dx_2 = 	\int_{-H}^{H}V_{\rm out}(x_2)dx_2=1$.\par
\end{itemize}
\vspace{1em}
The set of the admissible flow shapes with controlled norm $r>0$ is denoted by $\mathcal{F}_{r,U}$.
\end{definition}
\vspace{1em}
Physically, the controlled norm $r>0$ in Definition \ref{FU} measures the maximum strength of the flow by taking into account also its gradient.
By the compact embedding $W^{1,\infty}(-H,H)^2 \Subset L^\infty(-H,H)^2$ and the weak lower semi-continuity of the norm, the set $\mathcal{F}_{r,U}$ is closed in the weak-$\ast$ topology.
\vspace{1em}
\begin{proposition}\label{closed-flows}
	Let $U\in \{0,1\}$ and $r>0$. Assume that a sequence $\{(f_n, g_n)\}_{n\in \mathbb{N}}\subset\mathcal{F}_{r,U}$ converges weakly-$\ast$
	in $W^{1,\infty}(-H,H)^2$ to some $(f,g)\in W^{1,\infty}(-H,H)^2$. Then $(f, g)\in\mathcal{F}_{r,U}$.
\end{proposition}
\vspace{1em}
Next, we confine the bodies to a proper subregion of $R$. This choice is motivated by the applications to wind-bridges interactions:
the range of possible shapes of the cross-section of a bridge deck lies in a confined region, typically a rectangle. Moreover, in wind
tunnel experiments the deck has to be built far away from the walls $\partial R$.	
\vspace{1em}
\begin{definition}\label{def-admbodies}
	Let $D\subset R$ be a closed rectangle homothetic to $R$ and let $0<\alpha<|D|$.
	Let $\mathcal{C}_D$ be the set of nonempty compact convex domains contained in $D$.
	We say that $B\in \mathcal{C}_D$ is an admissible body shape if $|B|=\alpha$ and we then write $B\in\mathcal{C}_{\alpha, D}$.\par
	Let $d^H(\cdot,\cdot)$ denote the Hausdorff distance between two sets. We say that a family of sets $\{B_\eps\}_{\eps >0}\subset \mathcal{C}_{\alpha,D}$
	converges to $B\in  \mathcal{C}_{\alpha,D}$ as $\eps\rightarrow 0$ in the sense of Hausdorff if
	$$d^H(B_\eps, B)\rightarrow 0 \quad \mbox{as} \quad \eps\rightarrow 0\qquad\mbox{(notation: $B_\eps  \xrightarrow{H} B$).}$$
\end{definition}

Physically, $\alpha>0$ measures the amount of concrete used to build the cross-section of the deck.
\vspace{1em}
\begin{remark}
	In Definition \ref{def-admbodies}, the rectangle $D$ can be replaced by any compact subdomain of $R$ without
	altering the results in this paper. For instance, the fact that all the domains of the family $\{B_\eps\}_{\eps>0}$ are
	contained in a compact set $D\subset R$ still allows to obtain uniform bounds in the proofs of Theorems \ref{theo-cont-flows}
and \ref{theo-cont-bodies} below.
	Also the convexity request on the admissible body shapes can be relaxed by considering, instead,
	uniformly Lipschitz shapes since the next compactness result holds also within this class. Convexity is assumed in order
	to avoid pathological limit configurations with no interest for applications in wind-bridges interactions.
\end{remark}
\vspace{1em}
This class of admissible body shapes has a crucial compactness property:
\vspace{1em}
\begin{proposition}\label{compact-bodies}
	Let $D\subset R$ and $0<\alpha<|D|$ be as in Definition \ref{def-admbodies}, then $\mathcal{C}_{\alpha,D}$ is a compact subset of $\mathcal{C}_D$ for the Hausdorff metric.
\end{proposition}
\begin{proof} It follows by combining \cite[Section 2]{HenPie18} with \cite[Theorem 1.8.20]{schneider}.\end{proof}

In the next subsections, we prove some continuity results of the lift \eqref{lift-weak} with respect to admissible flow and body shapes.

\subsection{Continuity of the lift with respect to the flow}

The functional spaces needed to establish well-posedness for \eqref{bv-pb} are the Sobolev spaces $H^1_0(\Omega)$ of vector fields and
$$ H^1_*(\Omega)=\{ u\in H^1(\Omega) \ | \ u=0 \ \mbox{ on } \ \partial B\},$$
 defined as the closure of $C^\infty_0(\overline{R}\setminus B)$  with respect to the Dirichlet norm $\|\nabla \cdot \|_{L^2(\Omega)}$.
Since the Poincar\'{e} inequality holds in this space, see \cite{GazSpe20}, the Dirichlet semi-norm is a norm also in $H^1_*(\Omega)$.
We also introduce the associated Sobolev constants
\begin{equation}\label{sob-const}
\mathcal{S}_*(\Omega)=\min\limits_{u\in H^1_*(\Omega)\setminus \{0\}} \frac{\|\nabla u\|^2_{L^2(\Omega)}}{\|u\|^2_{L^4(\Omega)}},\qquad \mathcal{S}_0(\Omega)=\min\limits_{u\in H^1_0(\Omega)\setminus \{0\}} \frac{\|\nabla u\|^2_{L^2(\Omega)}}{\|u\|^2_{L^4(\Omega)}}.
\end{equation}
Note that  $\mathcal{S}_*(\Omega) \leq \mathcal{S}_0(\Omega)$ and that both constants are non-increasing with respect to the inclusion of domains. We then consider their subspaces of solenoidal vectors
\begin{equation*}
V_*(\Omega)= \{u\in H^1_*(\Omega) \ | \ \nabla \cdot v =0  \ \mbox{in} \  \Omega\}, \qquad V(\Omega)= \{u\in H^1_0(\Omega) \ | \ \nabla \cdot v =0 \ \mbox{in} \ \Omega\}
\end{equation*}
and we recall that for $f\in V_*(\Omega)\cup V(\Omega)$, $g\in H^1(\Omega)$ and  $h\in H^1_0(\Omega)$ we have
\begin{equation}\label{trilinear}
\int_{\Omega} (f \cdot \nabla g )\cdot h= -\int_{\Omega} (f \cdot \nabla h )\cdot g\, .
\end{equation}
The scalar pressure in \eqref{bv-pb} is defined up to the addition of a constant and, hence, will be sought within the space of
zero-mean functions
\begin{equation*}L^2_0(\Omega)=\left\{p\in L^2(\Omega) \ \bigg| \ \int_{\Omega} p \ dx =0\right\}.\end{equation*}
We may now state well-posedness of \eqref{bv-pb} for small flow magnitudes:
\vspace{1em}
\begin{proposition}\label{exi-uni-bv}
	Let $B\in \mathcal{C}_{\alpha,D}$, $\Omega= R\setminus B$ and $(V_{\rm in}, V_{\rm out})\in \mathcal{F}_{r,U}$ with $U\in \{0,1\}$.
	There exists $\Lambda=\Lambda(r,D)>0$ such that, for $\lambda\in [0,\Lambda]$,
	 \eqref{bv-pb} admits a unique weak solution $(u,p)\in V_*(\Omega)\times L^2_0(\Omega)$ with flow $\lambda( V_{\rm in},  V_{\rm out})$,
	satisfying both the bounds
	\begin{equation*}
\|\nabla u\|_{L^2(\Omega)} < \mathcal{S}_0(\Omega),\qquad	\|\nabla u\|_{L^2(\Omega)} \leq C \lambda,\qquad\forall\lambda\in[0,\Lambda],
	\end{equation*}
	for $\mathcal{S}_0(\Omega)$ as in \eqref{sob-const} and some $C=C(r,D)>0$. In particular, the first bound gives rise to a uniform bound with respect to $\Omega$.
\end{proposition}
\begin{proof} From \cite[Theorem 3.1]{GazSpe20} we know that there exists a unique solution to \eqref{bv-pb} satisfying the first bound of the proposition. Since $B\in \mathcal{C}_{\alpha,D}$, we have that $R\setminus D \subset \Omega$ and $\mathcal{S}_0(\Omega) \leq \mathcal{S}_0(R\setminus D)$, hence yielding a uniform bound.\par In \cite[Theorem 2.2]{BocGaz23} we showed uniqueness
for $\lambda\in[0, \Lambda]$ with $\Lambda=\Lambda(V_{\rm in}, V_{\rm out}, \Omega)>0$ and derived the bound
\begin{equation*}
	\|\nabla u\|_{L^2(\Omega)}\leq C \lambda
\end{equation*} for some $C=C(V_{\rm in}, V_{\rm out}, \Omega)>0$. More precisely, both $\Lambda$ and $C$ continuously depend on the $W^{1,\infty}(-H,H)^2$-norm of $(V_{\rm in},V_{\rm out})$ and, for $U=1$, also on the inverse of the Euclidean distance $d(\partial B, \partial R)$.
 Since $(V_{\rm in},V_{\rm out})\in \mathcal{F}_{r,U}$ and $B\in \mathcal{C}_{\alpha,D}$,  we have that $(V_{\rm in},V_{\rm out})$ is uniformly bounded in $W^{1,\infty}(-H,H)^2$ and $d(\partial B, \partial R)\geq d(\partial D, \partial R)$. Thus, we obtain a uniform uniqueness threshold $\Lambda(r,D)$ and the second bound of the proposition follows.\end{proof}

In view of Proposition \ref{exi-uni-bv}, unless otherwise specified, from now on by {\em solution} of \eqref{bv-pb} we intend a weak solution
$(u,p)\in V_*(\Omega)\times L^2_0(\Omega)$.\vspace{1em}

\begin{remark}
The uniqueness threshold $\Lambda=\Lambda(r,D)$ in Proposition \ref{exi-uni-bv} is a map defined on $\mathbb{R}_+\times\mathcal{C}_R$, where $\mathcal{C}_R$ is the set of nonempty homothetic rectangles contained in $R$ ordered by the inclusion of domains. We have that  $\Lambda(\cdot,D)$ is a decreasing function, with $\Lambda(r,D)\rightarrow 0$ as $r\rightarrow +\infty$, and that $\Lambda(r,\cdot)$ is a non-increasing function. Moreover, when $U=1$,  $\Lambda(r,D)\rightarrow 0$ as $D$ approaches $R$.
\end{remark}
\vspace{1em}
Next, we study how the lift $\mathcal{L}_B(u,p)$ in \eqref{lift-weak}, associated with the unique solution $(u,p)$ to \eqref{bv-pb}, behaves
if we modify the flow. In \cite[Proposition 5.1]{BocGaz23} we proved the Lipschitz continuity of the lift with respect to the flow magnitude. Here, we prove its continuity with respect to the admissible flow shapes.\par
Fix $B\in \mathcal{C}_{\alpha,D}$, $U\in \{0,1\}$ and $\lambda\in[0,\Lambda]$. Then consider the cascade of boundary-value problems in $\Omega=R\setminus B$
\begin{equation}\label{bound-value-dataeps}
\begin{aligned}
&- \Delta u_\eps + u_\eps\cdot \nabla u_\eps + \nabla p_\eps =0, \quad \nabla\cdot u_\eps=0 \quad \mbox{in} \quad \Omega,\\
{u_\eps}_{|_{\partial B}}={u_\eps}&_{|_{\Gamma_b}}=0, \quad {u_\eps}_{|_{\Gamma_t}}=\lambda U e_1, \quad {u_\eps}_{|_{\Gamma_l}}=\lambda V^\eps_{\rm in}e_1,\quad {u_\eps}_{|_{\Gamma_r}}=\lambda V^\eps_{\rm out}e_1,
\end{aligned}
\end{equation}
with $(V^\eps_{\rm in}, V^\eps_{\rm out})\in \mathcal{F}_{r,U}$ for any $\eps>0$. We then show the continuity of the lift with respect to the weak-$\ast$ convergence of the flow shapes.
\vspace{1em}
\begin{theorem}\label{theo-cont-flows}
Let $B\in \mathcal{C}_{\alpha,D}$, $\Omega= R\setminus B$, $(V_{\rm in},V_{\rm out})\in \mathcal{F}_{r,U}$ with $U\in \{0,1\}$, and $\Lambda$ as in Proposition
\ref{exi-uni-bv}. Given $\lambda\in[0,\Lambda]$, let $(u,p)\in V_*(\Omega)\times L^2_0(\Omega)$ be the unique solution to \eqref{bv-pb}.
If $\{(V^\eps _{\rm in}, V^\eps _{\rm out})\}_{\eps>0}\subset \mathcal{F}_{r,U}$ weakly-$\ast$ converges to $(V _{\rm in}, V _{\rm out})$  in $W^{1,\infty}(-H,H)^2 $ as $\eps \rightarrow 0$, then
\eqref{bound-value-dataeps} admits a unique solution $(u_\eps, p_\eps)\in V_*(\Omega)\times L^2_0(\Omega)$ for any $\eps>0$ and
\begin{equation}\label{uni-conv-lift}
\lim\limits_{\eps\rightarrow 0} \| \mathcal{L}_B (u_\eps, p_\eps) - \mathcal{L}_B (u, p)\|_{C([0,\Lambda])}=0.
\end{equation}
\begin{proof} Using the compact embedding $W^{1,\infty}(-H,H)^2\Subset H^{1/2}(-H,H)^2$, we have that $(V^\eps _{\rm in}, V^\eps _{\rm out})$ strongly converges to $(V_{\rm in},V_{\rm out})$  in $H^{1/2}(-H,H)^2 $ as $\eps\to0$. Replying the same estimates as in the proof of \cite[Theorem 3.5]{GazSpe20} yields the existence of a unique solution $(u_\eps, p_\eps)$ to \eqref{bound-value-dataeps} for $\lambda\in [0,\Lambda]$, see also Proposition \ref{exi-uni-bv}. 	
Moreover, there exists a constant $C=C(r,D)>0$ such that
		\begin{equation*}
		\|\nabla (u-u_\eps)\|_{L^2(\Omega)} +\|p-p_\eps\|_{L^2(\Omega)} \leq C \lambda\eps\leq C \Lambda\eps.
		\end{equation*}
Since $\Lambda$ is independent of $\eps$ (again by Proposition \ref{exi-uni-bv}), we infer that
		\begin{equation}\label{unilim-up}
			\lim\limits_{\eps\rightarrow 0}  \left(\|\nabla (u-u_\eps)\|_{L^2(\Omega)} +\|p-p_\eps\|_{L^2(\Omega)}\right)=0
			\end{equation}
uniformly with respect to $\lambda\in [0,\Lambda]$.
By linearity of the lift \eqref{lift-weak} with respect to $(u,p)$ and by continuity of the traces, we use \eqref{unilim-up} to deduce \eqref{uni-conv-lift}. Note that here $\mathcal{L}_B(u,p)$ does not vanish as in \cite[Theorem 3.7]{GazSpe20} since we deal with configurations that may be asymmetric with respect to the line $x_2=0$.
	\end{proof}
\end{theorem}

\subsection{Continuity of the lift with respect to the body shape}

We study here how the lift $\mathcal{L}_B(u,p)$ in \eqref{lift-weak}, associated with the unique solution $(u,p)$ to \eqref{bv-pb}, behaves if we perturb the body shape but {\it not} its measure. Consider a family of domains $\{B_\eps\}_{\eps>0}\subset \mathcal{C}_{\alpha,D}$ and the related problems, set in $\Omega_\eps=R\setminus B_\eps$:
\begin{equation}\label{bound-value-eps}
\begin{aligned}
&- \Delta u_\eps + u_\eps\cdot \nabla u_\eps + \nabla p_\eps =0, \quad \nabla\cdot u_\eps=0 \quad \mbox{in} \quad \Omega_\eps,\\
{u_\eps}_{|_{\partial B_\eps}}={u_\eps}_{|_{\Gamma_b}}&=0, \quad {u_\eps}_{|_{\Gamma_t}}=\lambda U e_1, \quad {u_\eps}_{|_{\Gamma_l}}=\lambda V_{\rm in}e_1,\quad {u_\eps}_{|_{\Gamma_r}}=\lambda V_{\rm out}e_1.
\end{aligned}
\end{equation}

We prove the continuity of the lift with respect to the Hausdorff convergence of the body shapes on which the lift acts.
\vspace{1em}
\begin{theorem}\label{theo-cont-bodies}
	Consider a family $\{B_{\eps}\}_{\eps>0}\subset \mathcal{C}_{\alpha,D}$ and $B\in\mathcal{C}_{\alpha,D}$ such that $B_\eps \xrightarrow{H}B$ as $\eps\rightarrow 0$. Let $\Omega_\eps= R\setminus B_\eps$, $\Omega=R\setminus B$, $(V_{\rm in}, V_{\rm out})\in \mathcal{F}_{r,U}$ with $U\in \{0,1\}$ and $\Lambda$ as in Proposition \ref{exi-uni-bv}. Given $\lambda\in[0,\Lambda]$, let $(u,p)\in V_*(\Omega)\times L^2_0(\Omega)$ be the unique solution to \eqref{bv-pb}. Then, for any $\eps >0$, \eqref{bound-value-eps} admits a unique solution $(u_\eps, p_\eps)\in V_*(\Omega_\eps)\times L^2_0(\Omega_\eps)$  with $\lambda$ in the same interval.  Moreover,
\begin{equation}\label{contlift}
\lim\limits_{\eps\rightarrow 0}\mathcal{L}_{B_\eps}(u_\eps, p_\eps) =\mathcal{L}_{B}(u, p).
\end{equation}
\end{theorem}

\begin{proof}
Since $\{B_\eps\}_{\eps>0} \subset \mathcal{C}_{\alpha,D}$ for any $\eps>0$, Proposition \ref{exi-uni-bv} yields the existence of a unique solution $(u_\eps, p_\eps)\in V_*(\Omega_\eps)\times L^2_0(\Omega_\eps)$ to \eqref{bound-value-eps} for any $\lambda\in[0,\Lambda]$ and the related uniform bound with respect to $\eps$.\par
We study the convergence of $(u_\eps, p_\eps)$ to $(u, p)$ as $\eps\rightarrow 0$. Let us first consider the case when $B$ is a domain of class $C^{1,1}$.
Then we know from \cite[Theorem 2.2]{BocGaz23} that the unique solution to \eqref{bv-pb} is a \emph{strong} solution, that is,
$(u,p)\in H^2(\Omega)\times H^1(\Omega)$.\par
Since $B_\eps \xrightarrow{H}B$ without an ``order'' (for instance, $B_\eps$ outer-approximating $B$ as in \cite{GazSpe20}), the solution $(u,p)$ may not be
defined in $\Omega_\eps$ and/or the solution $(u_\eps,p_\eps)$ may not be defined in $\Omega$. In order to compare $u$ and $u_\eps$, we extend them by zero, respectively, in $B$ and $B_\eps$, so they are both defined in the whole rectangle $R$: we maintain the same notation for their extension.
In particular,
\begin{equation}\label{restriction}
u\equiv0\mbox{ in }\Omega_\eps\setminus \Omega\quad\mbox{and}\quad\|u\|_{H^{1/2}(\partial B_\eps \setminus \Omega)}=0.
\end{equation}
Arguing as in \cite{GazSpe20} and using the embedding $H^2(\Omega)\hookrightarrow C^{0,\nu}(\overline{\Omega})$ for any $\nu\in (0,1)$, we obtain
the uniform continuity of $u$ in $\overline\Omega$ which guarantees that there exists $\sigma_\eps>0$, with $\sigma_\eps\rightarrow 0$ as $\eps \rightarrow 0$,
such that $\|u\|_{H^{1/2}(\partial B_\eps\cap \Omega)}\leq \sigma_\eps$. Combined with \eqref{restriction}, this implies that
$\|u\|_{H^{1/2}(\partial B_\eps)}\leq \sigma_\eps$.
 Moreover, there exists a solenoidal vector field $s_\eps\in H^1(\Omega_\eps)$ and a constant $C>0$, independent of $\eps$, such that
	\begin{equation}\label{sol-ext-eps}
 {s_\eps}_{|_{\partial R}}= 0, \qquad {s_\eps}_{|_{\partial B_\eps}}= u_{|_{\partial B_\eps}}, \qquad \|\nabla s_\eps\|_{L^2(\Omega_\eps)}\leq C \|u\|_{H^{1/2}(\partial B_\eps)}\leq C \sigma_\eps.
	\end{equation}
Let $v:=u-u_\eps -s_\eps\in V(\Omega_\eps)$. Since $u$ and $u_\eps$ weakly solve, respectively, \eqref{bv-pb} and \eqref{bound-value-eps}, we infer that $v$ satisfies
\begin{equation}\label{weak-formu}
	\int_{\Omega_\eps} \nabla v:\nabla \varphi + \int_{\Omega_\eps} (v\cdot \nabla v) \cdot \varphi  + 	\int_{\Omega_\eps} \nabla s_\eps:\nabla \varphi + \int_{\Omega_\eps}f_\eps\cdot \varphi =0 \qquad \forall\varphi\in V(\Omega_\eps),
\end{equation}
where $f_\eps= v\cdot \nabla (u_\eps+s_\eps) + (u_\eps+ s_\eps) \cdot \nabla v - s_\eps\cdot \nabla s_\eps$.
Taking $\varphi=v$ in \eqref{weak-formu} and using \eqref{trilinear} yields
\begin{equation}\label{test-weak-v}
\|\nabla v\|^2_{L^2(\Omega_\eps)}=-\int_{\Omega_\eps} \nabla s_\eps : \nabla v -\int_{\Omega_\eps} (v\cdot \nabla u )\cdot v + \int_{\Omega_\eps} (s_\eps\cdot \nabla s_\eps) \cdot v.
\end{equation}
Since $u\equiv 0$ in $\Omega_\eps\setminus \Omega$, we have the uniform bounds
\begin{equation}\label{uni-bound-u}
\| \nabla u\|_{L^2(\Omega_\eps)}\leq\|\nabla u \|_{L^2( \Omega)},\qquad\|u\|_{L^4(\Omega_\eps)}\leq \|u\|_{L^4(\Omega)}.
\end{equation}
By the H\"{o}lder inequality, the continuous embeddings $ H^1_0, H^1_* \subset L^4$ with constants \eqref{sob-const} and the bounds \eqref{uni-bound-u}, we estimate the right-hand side of \eqref{test-weak-v} by
	\begin{align*}
	&\left|\int_{\Omega_\eps}\nabla s_\eps: \nabla v\right|\leq \|\nabla s_\eps\|_{L^2(\Omega_\eps)}\|\nabla v\|_{L^2(\Omega_\eps)},\\[7pt]
&	\left|\int_{\Omega_\eps} (v\cdot \nabla u)\cdot v \right|\leq \|\nabla u\|_{L^2(\Omega_\eps)}\|v\|^2_{L^4(\Omega_\eps)}\leq \frac{1}{\mathcal{S}_0(\Omega_\eps)}\|\nabla u\|_{L^2(\Omega)}\|\nabla v\|^2_{L^2(\Omega_\eps)},\\[7pt]
	&\left|\int_{\Omega_\eps} (s_\eps \cdot\nabla s_\eps) \cdot v \right|\leq \left(\|s_\eps-u\|_{L^4(\Omega_\eps)} + \|u\|_{L^4(\Omega_\eps)}\right)\|\nabla s_\eps\|_{L^2(\Omega_\eps)} \| v\|_{L^4(\Omega_\eps)}\\[2pt]& \leq \frac{1}{\sqrt{\mathcal{S}_0(\Omega_\eps)}}\left(\frac{\|\nabla s_\eps\|_{L^2(\Omega_\eps)} + \|\nabla u\|_{L^2(\Omega_\eps)}}{\sqrt{\mathcal{S}_*(\Omega_\eps)}}+ \|u\|_{L^4(\Omega)}\right)\|\nabla s_\eps\|_{L^2(\Omega_\eps)}\|\nabla v\|_{L^2(\Omega_\eps)}
	\\[2pt]& \leq \frac{1}{\sqrt{\mathcal{S}_*(\Omega_\eps)}}\bigg(\frac{\|\nabla s_\eps\|_{L^2(\Omega_\eps)}+\|\nabla u\|_{L^2(\Omega)}}{\sqrt{\mathcal{S}_*(\Omega_\eps)}} +  \frac{\|\nabla u\|_{L^2(\Omega)}}{\sqrt{\mathcal{S}_*(\Omega)}}\bigg)\|\nabla s_\eps\|_{L^2(\Omega_\eps)}\|\nabla v\|_{L^2(\Omega_\eps)}.	
 	\end{align*}
Whence, \eqref{test-weak-v} yields
\begin{eqnarray*}
\|\nabla v\|_{L^2(\Omega_\eps)} &\le& \frac{\|\nabla u\|_{L^2(\Omega)}}{\mathcal{S}_0(\Omega_\eps)}\, \|\nabla v\|_{L^2(\Omega_\eps)}\\
\ &\ & +\left[1+
\frac{\|\nabla s_\eps\|_{L^2(\Omega_\eps)}+\|\nabla u\|_{L^2(\Omega)}}{\mathcal{S}_*(\Omega_\eps)} +
\frac{\|\nabla u\|_{L^2(\Omega)}}{\sqrt{\mathcal{S}_*(\Omega)\mathcal{S}_*(\Omega_\eps)}}\right]\|\nabla s_\eps\|_{L^2(\Omega_\eps)}\, .
\end{eqnarray*}
We know from \cite{HenPie18} that the embedding constants in \eqref{sob-const} are continuous with respect to the Hausdorff convergence of domains.
In particular, there exists $\delta_\eps>0$, with $\delta_\eps \searrow 0$ as $\eps\rightarrow 0$, such that $\mathcal{S}_0(\Omega_\eps)\geq \mathcal{S}_0(\Omega)-\delta_\eps$ and $\mathcal{S}_*(\Omega_\eps)\geq \mathcal{S}_*(\Omega)-\delta_\eps$.
It follows from \eqref{sol-ext-eps} and the previous estimate that
	\begin{equation}\label{est-gradv}
	\begin{aligned}
	&\left(	1 -\frac{\|\nabla u\|_{L^2(\Omega)}}{\mathcal{S}_0(\Omega)-\delta_\eps} \right)  \|\nabla v\|_{L^2(\Omega_\eps)} \\[2pt]& \leq C \left(1+ \frac{C\sigma_\eps+ \|\nabla u\|_{L^2(\Omega)}}{S_*(\Omega)-\delta_\eps} + \frac{\|\nabla u\|_{L^2(\Omega)}}{\sqrt{(S_*(\Omega)-\delta_\eps)S_*(\Omega)}}\right)\sigma_\eps.
	\end{aligned}
	\end{equation}

By the first bound in Proposition \ref{exi-uni-bv}, there exists $\eps_0>0$ such that, for any $\eps\in(0,\eps_0)$, the parenthesis in the left-hand side of \eqref{est-gradv} is uniformly strictly positive  and the parenthesis in the right-hand side is uniformly bounded with respect to $\eps$. Using again \eqref{sol-ext-eps} and the fact that $\sigma_\eps\rightarrow 0$ as $\eps\rightarrow 0$, we infer  that  \begin{equation}\label{conv-dirichlet}\|\nabla (u_\eps-u)\|_{L^2(\Omega_\eps)}\leq 	\|\nabla s_\eps\|_{L^2(\Omega_\eps)}+ \|\nabla v\|_{L^2(\Omega_\eps)} \rightarrow 0 \quad \mbox{as}\quad \eps\rightarrow 0.\end{equation}
Furthermore, by arguing as in \cite{GazSpe20} we know that there exists $C>0$, independent of $\eps$, such that
\begin{equation*}
\|p_\eps-p\|_{L^2(\Omega_\eps)}\leq C(1+\|\nabla u_\eps\|_{L^2(\Omega_\eps)} + \|\nabla u\|_{L^2(\Omega_\eps)} )\|\nabla(u_\eps-u)\|_{L^2(\Omega_\eps)}.
\end{equation*}
Thanks to \eqref{uni-bound-u}, the uniform bound in Proposition \ref{exi-uni-bv} and the convergence \eqref{conv-dirichlet}, we then conclude that
\begin{equation}\label{continuity-u-p}
\lim\limits_{\eps\rightarrow 0}\left(\|\nabla (u_\eps-u)\|_{L^2(\Omega_\eps)}+\|p_\eps - p\|_{L^2(\Omega_\eps)}\right)=0.
\end{equation}

Next, we focus on the lift \eqref{lift-weak}. We rewrite it as \begin{equation}\label{lift-diff}
	\mathcal{L}_{B}(u,p)= e_2 \cdot \langle \mathbb{T}(u,p)n,1 \rangle_{\partial R}-e_2 \cdot \langle \mathbb{T}(u,p)n,1 \rangle_{\partial \Omega}
	\end{equation}
where, by arguing as in the proof of \cite[Theorem 2.2]{BocGaz23},
the first term in \eqref{lift-diff} can be treated as an integral. By linearity of the stress tensor, we obtain
$$\begin{aligned}
	&	\mathcal{L}_{B_\eps}(u_\eps,p_\eps)-\mathcal{L}_{B}(u,p)\\[2pt]
&=e_2\cdot\left( \int_{\partial R}\mathbb{T}(u_\eps-u,p_\eps-p)n - \int_{\Omega_\eps}\nabla \cdot \mathbb{T} (u_\eps,p_\eps)+\int_{\Omega}\nabla \cdot \mathbb{T} (u,p)\right)\\[2pt]
&= e_2\cdot\int_{\partial R}\mathbb{T}(u_\eps-u,p_\eps-p) n - \int_{\Omega_\eps}u_\eps \cdot \nabla u_{\eps,2} + \int_{\Omega}u\cdot \nabla u_2.\\[2pt]
	\end{aligned}
$$
Moreover, since $u\in V_*(\Omega)$ and $u_\eps\in V_*(\Omega_\eps)$, an integration by parts gives
\begin{equation*}
\int_{\Omega} u\cdot \nabla u_2-\int_{\Omega_\eps}u_\eps\cdot\nabla u_{\eps,2} = \int_{\partial R} u_2 u\cdot n -\int_{\partial R} u_{\eps,2} u_\eps\cdot n = 0,
\end{equation*}
where the last equality follows from the fact that $u=u_\eps$ on $\partial R$. Therefore,
\begin{equation}\label{diffshape-lift}
\mathcal{L}_{B_\eps}(u_\eps,p_\eps)-\mathcal{L}_{B}(u,p)=e_2\cdot\int_{\partial R}\mathbb{T}(u_\eps-u,p_\eps-p) n
\end{equation}
and we now estimate the term on the right-hand side of \eqref{diffshape-lift}.

Since $B_\eps \xrightarrow{H}B$ as $\eps\rightarrow 0$, there exists a domain $\Omega_0\subset \Omega_\eps$  with  interior boundary of class $C^{1,1}$ such that $\partial \Omega_0\cap \partial \Omega_\eps=\partial R$ for any $\eps\in (0,\eps_0)$ (with a possibly smaller $\eps_0$). From \cite[Theorem 2.2]{BocGaz23} we know that $(u,p)$ and $(u_\eps, p_\eps)$ belong to $H^2(\Omega_0)\times H^1(\Omega_0)$ for any $\eps\in (0, \eps_0)$ and,
using again \cite[Theorem 2.2]{BocGaz23} and the convergence \eqref{continuity-u-p}, we have that $\|u_\eps\|_{H^2(\Omega_0)}$ and $\|p_\eps\|_{H^1(\Omega_0)}$ are uniformly bounded with respect to $\eps \in (0,\eps_0)$. This implies that also $\|u_\eps- u\|_{H^2(\Omega_0)}$ and $\|p_\eps- p\|_{H^1(\Omega_0)}$ are uniformly bounded with respect to $\eps \in (0,\eps_0)$. We combine this fact with \eqref{continuity-u-p} and, by interpolation, we obtain that $u_\eps \rightarrow u$
in $H^{3/2}(\Omega_0)$ and  $p_\eps \rightarrow p$ in $H^{1/2+r}(\Omega_0)$ with $r>0$ as $\eps\rightarrow 0$. From trace-regularity results (see \cite{Gag57,Galdi-steady}) we know that
\begin{equation}\label{conv-traces}
u_\eps \rightarrow u \quad \mbox{in} \quad  H^1(\partial R) \qquad \mbox{and} \qquad 	p_\eps \rightarrow p \quad \mbox{in} \quad L^2(\partial R) \quad\mbox{as}\quad \eps\rightarrow 0 ,
\end{equation}
and, using the Hölder inequality, we estimate the right hand side in \eqref{diffshape-lift} by
\begin{equation*}\begin{aligned}
\left|e_2\cdot\int_{\partial R}\mathbb{T}(u_\eps-u,p_\eps-p) n\right|
\leq C \left(\| u_\eps-u\|_{H^1(\partial R)} + \|p_\eps-p\|_{L^2(\partial R)}\right)
\end{aligned}
\end{equation*}
with $C>0$ independent of $\eps$. Thanks to \eqref{conv-traces}, we finally conclude that \eqref{contlift} holds.\par	
If $B$ is not a $C^{1,1}$-domain, we use a density argument. From the previous analysis, we know that, given $\tau>0$, there exists $\delta>0$ such that, for
a domain $K_1\in\mathcal{C}_{\alpha,D}$ and a $C^{1,1}$-domain $K_2\in\mathcal{C}_{\alpha,D}$ with $d^H(K_1, K_2)<\delta$, we have
$$|\mathcal{L}_{K_1}(u_{K_1},p_{K_1})-\mathcal{L}_{K_2}(u_{K_2},p_{K_2})|< \frac{\tau}{2},$$
	where $(u_{K_i},p_{K_i})$ is the unique solution to \eqref{bound-value-eps}  in $\Omega_i = R\setminus K_{i}$ for $i=1,2$.
	Taking $\eps$ sufficiently small such that $d^H(B_\eps, B)< \delta,$ it is possible to construct a  $C^{1,1}$-domain $K_\eps \in \mathcal{C}_{\alpha,D}$ with  $d^H(B_\eps, K_\eps)<\delta$ and $d^H(K_\eps, B)<\delta$. It follows that
\begin{equation*}\begin{aligned}
&|\mathcal{L}_{B_\eps}(u_\eps, p_\eps)-\mathcal{L}_B(u,p) | \\
&\leq |\mathcal{L}_{B_\eps}(u_\eps, p_\eps)-\mathcal{L}_{K_\eps}(u_{K_\eps},p_{K_\eps}) | + |\mathcal{L}_{K_\eps}(u_{K_\eps},p_{K_\eps}) - \mathcal{L}_{B}(u, p) |< \frac{\tau}{2} +\frac{\tau}{2} = \tau
	\end{aligned}\end{equation*}
and \eqref{contlift} holds also if $B$ is not $C^{1,1}$.\end{proof}

In the next two sections, we present two different applications of the continuous dependence results proved in this section. The first can be used to study stability of artificial configurations in wind tunnels, the second has a direct application in the stability analysis of real-life configurations.

\section{Asymmetry in wind tunnel experiments}\label{sec-windtunnel}

In this section we introduce the FSI problem obtained by coupling the boundary-value problem \eqref{bv-pb} with the zero-lift condition, that is,
\begin{equation}\label{FSI}
\begin{aligned}
&- \Delta u + u\cdot \nabla u + \nabla p =0, \quad \nabla\cdot u=0 \quad \mbox{in} \quad \Omega,\\
{u}_{|_{\partial B}}={u}_{|_{\Gamma_b}}&=0, \quad {u}_{|_{\Gamma_t}}=\lambda U e_1, \quad {u}_{|_{\Gamma_l}}=\lambda V_{\rm in}e_1,\quad {u}_{|_{\Gamma_r}}=\lambda V_{\rm out}e_1,\\
&\qquad\qquad \qquad \quad \mathcal{L}_B(u,p)=0,\end{aligned}\tag{FSI}
\end{equation}with $\mathcal{L}_B(u,p)$ as in \eqref{lift-weak}. We know from \cite{GazSpe20,BocGaz23} that, when the body shape is symmetric (with respect to the line $x_2=0$) and the flow shape is even, \eqref{FSI} admits a unique solution for any $\lambda\in [0, \Lambda]$. However, for asymmetric flow shapes, the unique solution provided by Proposition \ref{exi-uni-bv} may create a nonzero-lift on bodies with either symmetric or asymmetric shapes.
Our goal is to show the existence of asymmetric body shapes and asymmetric flow shapes such that \eqref{FSI} admits a unique (zero-lift) solution.
Although a ``global'' zero-lift result (valid {\em for any} $\lambda\in[0,\Lambda]$) appears out of reach, we prove its ``local'' version, with {\em fixed} flow magnitude $\lambda\in[0,\Lambda]$. The flow magnitude is
usually fixed in wind tunnel experiments, where engineers create ad-hoc flows that, typically, have top and bottom zero boundary conditions. Thus, throughout this section we merely consider (FSI) with $U=0$ and we prove
\vspace{1em}
\begin{theorem}\label{theo-zerolift-loc}
Let $0<\alpha<|D|$, $r>0$, $U=0$, and let $\Lambda$ be as in Proposition \ref{exi-uni-bv}. For any $\lambda \in[0,\Lambda]$ there exist an asymmetric body shape $\widetilde{B}\in \mathcal{C}_{\alpha, D}$ and/or a non-even flow shape $(\widetilde{V}_{\rm in}, \widetilde{V}_{\rm out})\in \mathcal{F}_{r,0}$, such that \eqref{FSI} admits a unique solution in $\widetilde{\Omega}=R\setminus \widetilde{B}$ with flow $\lambda(\widetilde{V}_{\rm in}, \widetilde{V}_{\rm out})$.
\end{theorem}

\begin{proof} Fix $\lambda\in[0,\Lambda]$.
Take any asymmetric $B\in\mathcal{C}_{\alpha, D}$ and any non-even flow shape $(V_{\rm in}, V_{\rm out})\in \mathcal{F}_{r,0}$. We know from Proposition \ref{exi-uni-bv} that
\eqref{bv-pb} admits a unique solution $(u,p)$ in $\Omega=R\setminus B$ with flow $\lambda(V_{\rm in}, V_{\rm out})$. Then we compute the associated lift $\mathcal{L}_B(u,p)$ in \eqref{lift-weak}.\par
%	When $B$ is asymmetric and $(V_{\rm in}, V_{\rm out})$  are non-even functions, two cases are possible:
	 If $\mathcal{L}_B(u,p)=0$, then the proof is completed by taking $\widetilde{B}=B$ and $(\widetilde{V}_{\rm in}, \widetilde{V}_{\rm out})=(V_{\rm in}, V_{\rm out})$. If $\mathcal{L}_B(u,p)\neq 0$, say $\mathcal{L}_B(u,p)> 0$,  we consider the reflection of $B$ with respect to the $x_1$-axis, that is,
	\begin{equation*}B_{\rm ref}=\{(x_1,-x_2)\in R;\ (x_1,x_2)\in B\}\end{equation*}
	and the flow shape reflection $(V^{\rm ref}_{\rm in}, V^{\rm ref}_{\rm out})$ defined by
\begin{equation*}
(V^{\rm ref}_{\rm in}(x_2), V^{\rm ref}_{\rm out}(x_2))=(V_{\rm in}(-x_2), V_{\rm out}(-x_2))\quad  \mbox{for all} \quad x_2\in [-H,H].
\end{equation*}
By exploiting the symmetries, see \cite{GazSpe20}, problem \eqref{bv-pb} in $\Omega_{\rm ref}=R\setminus B_{\rm ref}$ with flow
$\lambda(V^{\rm ref}_{\rm in}, V^{\rm ref}_{\rm out})$ admits a unique solution $(u_{\rm ref},p_{\rm ref})$ and
\begin{equation}\label{lift-neg}
\mathcal{L}_{B_{\rm ref}}(u_{\rm ref},p_{\rm ref})=	-\mathcal{L}_{B}(u,p)<0.
\end{equation}

The strategy of the proof consists in constructing a sequence of domains that goes from $B$ to  $B_{\rm ref}$ avoiding symmetric domains and a sequence of flow shapes that goes from $(V_{\rm in}, V_{\rm out}) $ to $(V^{\rm ref}_{\rm in}, V^{\rm ref}_{\rm out})$ avoiding even functions.\par
	We first construct a continuous family of domains $\{B_\eps\}_{\eps\in[0,1]}$ such that\par\vspace{1em}
(i) $B_0=B$ and $B_1=B_{\rm ref}$,\hfill(ii) $B_\eps$ is asymmetric for any $\eps\in [0,1]$,\par
(iii) $d^H(B_{\eps}, B_{\overline{\eps}})\rightarrow 0$ as $\eps\rightarrow \overline{\eps}$ for any  $\overline{\eps} \in[0,1]$,\hfill
(iv) $B_\eps \in \mathcal{C}_{\alpha,D}$ for any $\eps\in[0,1]$.\par
\vspace{1em}
We present the details of the construction for a particular body shape $B$, but the method can be easily adapted to other shapes. Given $0<h<l<L$ and $h<H$, let $B$ be the right trapezium being the union of the closed rectangle
	$r=[-l,l]\times[-h,h]$ and the (closed) right triangle $T_0$ having vertices in $(l,-h)$, $(l,h)$, $(l+\gamma,h)$ for some $0<\gamma<l$.
All these parameters are chosen in such a way that $B\in \mathcal{C}_{\alpha,D}$ and the below construction maintains the inclusion
$B_\eps \in \mathcal{C}_{\alpha,D}$. For any $\eps\in\left[0,\tfrac{2}{3}\right]$ we modify $T_0$ with the right trapezium $T_\eps$ having vertices in
		$$
		(l,-h),\quad(l,h),\quad\left(l+\tfrac{\gamma}{1+\eps},h\right),\quad\left(l+\tfrac{\gamma}{1+\eps},h(1-2\eps)\right)
		$$
		so that $T_\eps$ degenerates into the original triangle for $\eps=0$, the area is maintained, \textit{i.e} $|T_\eps|=|T_0|=h\gamma$ for all $\eps\in\left[0,\tfrac{2}{3}\right]$,
		and $B_\eps=r\cup T_\eps$ is a convex pentagon with $|B_\eps|=\alpha$. In particular, $B_{2/3}$ has vertices in
		$$
		(-l,-h),\quad(-l,h),\quad\left(l+\tfrac{3\gamma}{5},h\right),\quad\left(l+\tfrac{3\gamma}{5},-\tfrac{h}{3}\right),\quad(l,-h).
		$$
		Then, for any $\eps\in\left[\tfrac{2}{3},1\right]$ we modify $T_{2/3}$ through a family of quadrilaterals $Q_\eps$ having vertices in
		$$
		(l,-h),\quad(l,h),\quad\left(l+\frac{9\gamma(1-\eps)}{15\eps-9\eps^2 -1},h\right),\quad\left(l+\frac{\gamma(6\eps-1)}{15\eps-9\eps^2 -1},(1-2\eps)h\right).
		$$
		Notice that $15\eps-9\eps^2 -1>0$ and $|Q_\eps|=|T_{2/3}|=h\gamma$ for all $\eps\in\left[\tfrac{2}{3},1\right]$. Thus, $B_\eps = r \cup Q_\eps$ is a convex hexagon with $|B_\eps|=\alpha$.
	At the end, $B_1=B_{\rm ref}$ is the symmetric of the original trapezium $B_0=B$. See Figure \ref{body-eps}
	for a qualitative construction, for the four values $\delta\in\{0 ,\tfrac23 ,\tfrac34 ,1\}$.
We see that $B_0=B$ (left) ``Hausdorff-continuously'' transforms into $B_1=B_{\rm ref}$ (right) maintaining the asymmetric property for each $B_\eps$. It is straightforward to check that the properties (i)-(ii)-(iii)-(iv) are satisfied by this particular family
	of domains.
\begin{figure}[h]
\includegraphics[scale=0.8]{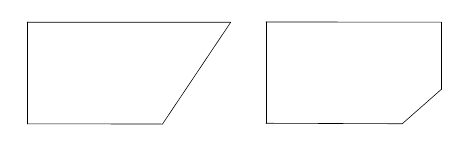}\includegraphics[scale=0.8]{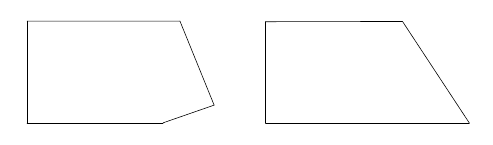}\vspace{0.5em}
\caption{The family of body shapes $\{B_\eps\}_{\eps\in[0,1]}$.}\label{body-eps}
\end{figure}

	 We have so constructed a path-connected family  of asymmetric body shapes $\{B_\eps\}_{\eps\in[0,1]}\subset \mathcal{C}_{\alpha,D}$ that starts in $B$ and ends in $B_{\rm ref}$.\par
We now construct a family of flow shapes $\{(V^\delta_{\rm in}, V^\delta_{\rm out})\}_{\delta\in[0,1]}$ such that
\vspace{1em}
\begin{itemize}
\item[(a)] $(V^0_{\rm in},V^0_{\rm out})=(V_{\rm in},V_{\rm out})$ and $(V^1_{\rm in},V^1_{\rm out})=(V^{\rm ref}_{\rm in},V^{\rm ref}_{\rm out})$,\par
\item[(b)] $ V^\delta_{\rm in}$ and $V^\delta_{\rm out}$ are non-even functions for any $\delta\in [0,1]$,\par
\item[(c)] $(V^{\delta}_{\rm in}, V^{\delta}_{\rm out}) \stackrel{\ast}{\rightharpoonup}(V^{\delta_1}_{\rm in}, V^{\delta_1}_{\rm out})$ in $W^{1,\infty}(-H,H)^2$ as $\delta\rightarrow\overline{\delta}$, for any $\overline{\delta}\in[0, 1]$,\par
\item[(d)] $(V^{\delta}_{\rm in}, V^{\delta}_{\rm out})\in\mathcal{F}_{r,0}$ for any $\delta\in [0,1]$.
\end{itemize}
\vspace{1em}
To this end, we write the flow shapes as the sum of their even and odd parts, namely,
$$
V_{\rm in, e}(x_2)=\frac{V_{\rm in, e}(x_2)+V_{\rm in, e}(-x_2)}{2},\quad V_{\rm in, o}(x_2)=\frac{V_{\rm in, o}(x_2)-V_{\rm in, o}(-x_2)}{2},\quad\forall|x_2|\le H,
$$
and, similarly, for $V_{\rm out}$. Hence, $V_{\rm in}=V_{\rm in, e} + V_{\rm in, o}$  and $V_{\rm out}=V_{\rm out, e} + V_{\rm out, o}$. We modify only the odd parts $V_{\rm in,o}$ and $V_{\rm out,o}$ in order to
reach $(-V_{\rm in, o}, -V_{\rm out, o})$, avoiding even functions. Therefore, at the end of the procedure, we have $V_{\rm in, e}-V_{\rm in, o}=V^{\rm ref}_{\rm in}$ and
$V_{\rm out, e}-V_{\rm out, o}=V^{\rm ref}_{\rm out}$.\par
Since $(V_{\rm in}, V_{\rm out})\in \mathcal{F}_{r,0}$, the odd parts vanish at $x_2=\pm H$. Focusing only on $V_{\rm in}$, two cases may occur:
either it has a unique zero in $(-H,H)$ at $x_2=0$ or it has (at least) three zeros, one at $x_2=0$ and two symmetric ones at $x_2=\pm x^*$,
with $x^*\in(0,H)$.
	In the case of (at least) three zeros, we define the modified odd part $V^\delta_{\rm in, o}$ by
	\begin{equation*}
	V^\delta_{\rm in, o}(x_2)=\begin{cases}
	(1-4\delta)V_{\rm in, o}(x_2) \qquad \qquad& \mbox{for} \ x^*\leq |x_2|\leq H,\\
	V_{\rm in, o}(x_2) \qquad \qquad & \mbox{for} \  |x_2|<  x^*
	\end{cases}\qquad\mbox{ for $\delta\in [0, \tfrac{1}{2}]$},
	\end{equation*}
	\begin{equation*}
	V^\delta_{\rm in, o}(x_2)=\begin{cases}
	-V_{\rm in, o}(x_2) \qquad \qquad& \mbox{for} \ x^*\leq |x_2|\leq H,\\
	(3-4\delta)V_{\rm in, o}(x_2) \qquad \qquad & \mbox{for} \  |x_2|<  x^*
	\end{cases}\qquad\mbox{ for $\delta\in (\tfrac{1}{2},1]$}.\vspace{1em}
\end{equation*}
See Figure \ref{odd-part-three} for a qualitative construction in the case $V_{\rm in, o}(x_2)=\sin\left(2\pi x_2\right)$ with $H=1$,
for the five values $\delta\in\{0,\tfrac14 ,\tfrac12 ,\tfrac34 ,1\}$.
\begin{figure}
		\centering
		\begin{subfigure}{0.3\textwidth}
			\includegraphics[width=\textwidth]{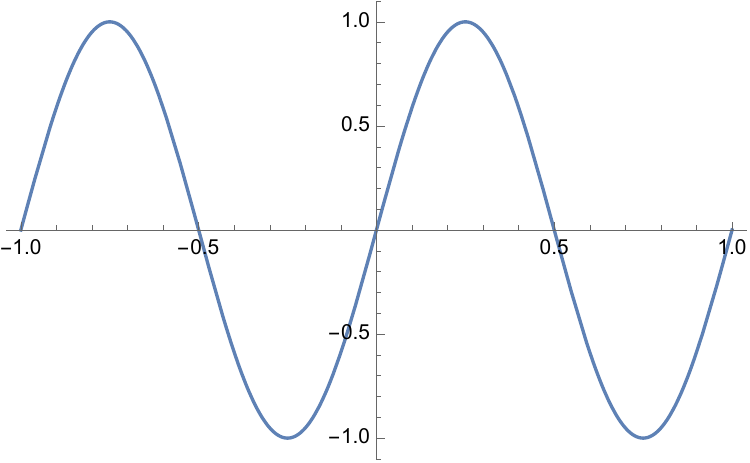}
		\end{subfigure}\hfill
		\begin{subfigure}{0.3\textwidth}
			\includegraphics[width=\textwidth]{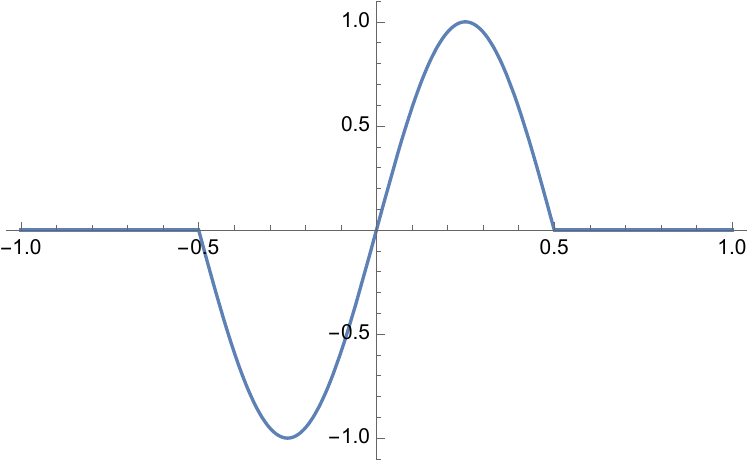}
		\end{subfigure}\vspace{1em}\hfill
		\begin{subfigure}{0.3\textwidth}
			\includegraphics[width=\textwidth]{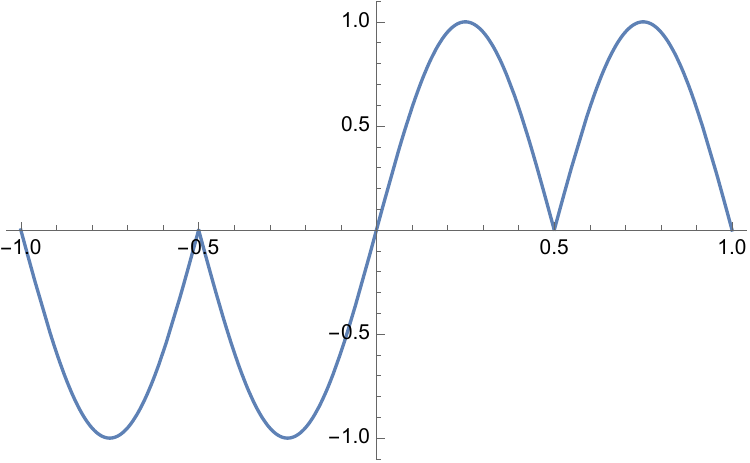}
		\end{subfigure}\hfill
		\begin{subfigure}{0.3\textwidth}
			\includegraphics[width=\textwidth]{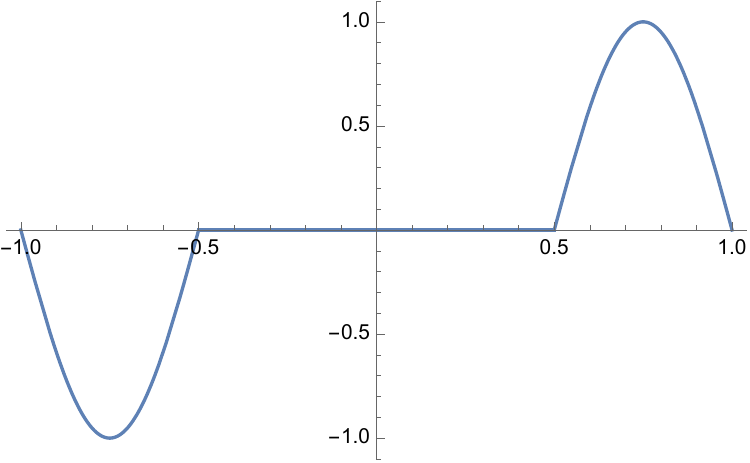}
		\end{subfigure}\qquad
\begin{subfigure}{0.3\textwidth}
	\includegraphics[width=\textwidth]{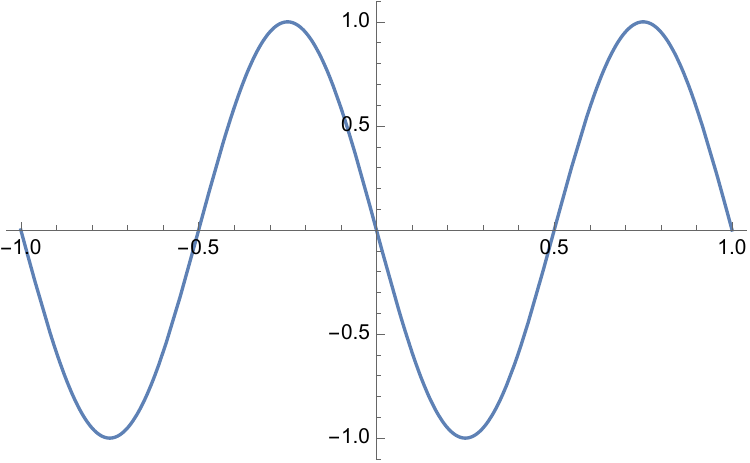}
\end{subfigure}\vspace{1em}
		\caption{The family of the flow-shape odd parts $\{V_{\rm in, o}^\delta\}_{\delta\in [0,1]}$ in the case of three zeros in $(-H,H)$.}
		\label{odd-part-three}
	\end{figure}
We  see that $V^0_{\rm in, o}= V_{\rm in, o}$ (above, left) ``$W^{1,\infty}$-continuously transforms'' into $V^1_{\rm in, o}=V^{\rm ref}_{\rm in, o}$ (below, right) maintaining the non-even property for each $V_{\rm in, o}^\delta$. In fact, with this procedure
$V^\delta_{\rm in, e}=V_{\rm in, e}$ and $V^\delta_{\rm in, o}$ is an odd function for any $\delta$. A similar construction can be repeated for $V^\delta_{\rm out}$
and the properties (a)-(b)-(c) are satisfied by this particular family of flow shapes.\par
When the odd part has a unique zero in $(-H,H)$, the strategy consists in creating two additional zeros at some $x_2=\pm x^*$, other than $x_2=0$,
and then use the symmetry-breaking argument as in the previous case. For the sake of simplicity, we omit the details of the construction and we refer to Figure \ref{odd-part-one} for a qualitative description in the case $V_{\rm in, o}(x_2)=\sin\left(\pi x_2\right)$ when $H=1$.
In order to transform $V^0_{\rm in, o}=V_{\rm in, o}$ (above, left) into $V^1_{\rm in, o}=V^{\rm ref}_{\rm in, o}$ (below, right), we
maintain the non even-property (in fact, odd) for each $V^{\delta}_{\rm in, o}$.
	\begin{figure}
			\centering
		\begin{subfigure}{0.3\textwidth}
			\includegraphics[width=\textwidth]{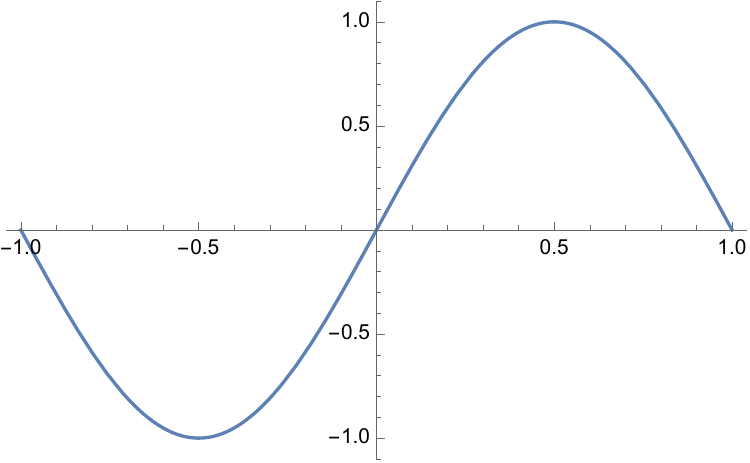}
		\end{subfigure}\hfill
		\begin{subfigure}{0.3\textwidth}
			\includegraphics[width=\textwidth]{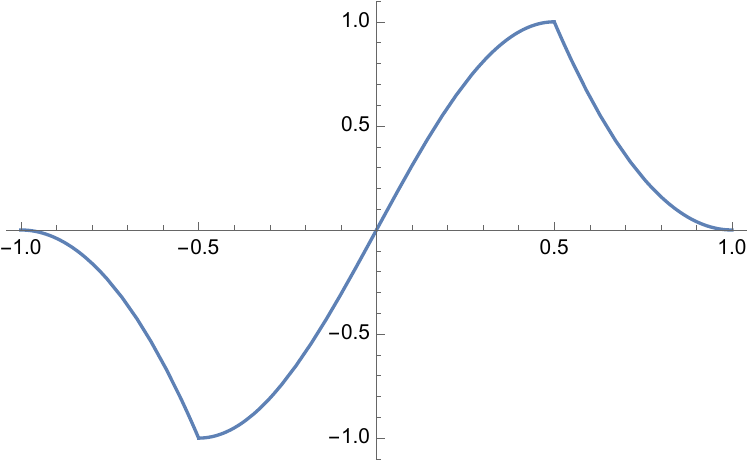}
		\end{subfigure}\vspace{1em}\hfill
		\begin{subfigure}{0.3\textwidth}
			\includegraphics[width=\textwidth]{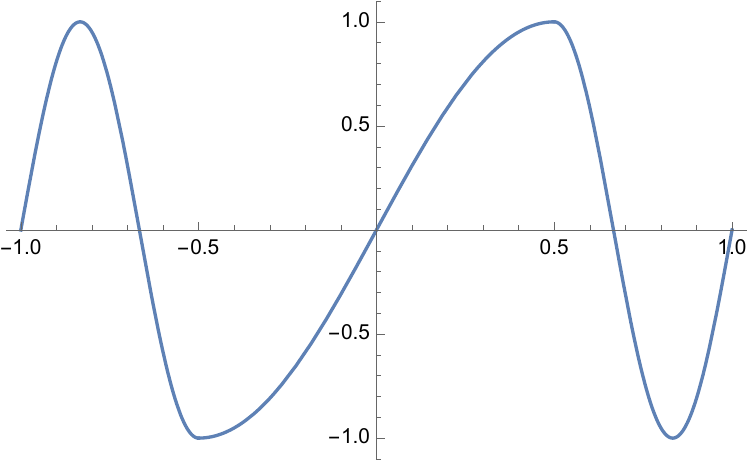}
		\end{subfigure}\\
		\begin{subfigure}{0.3\textwidth}
			\includegraphics[width=\textwidth]{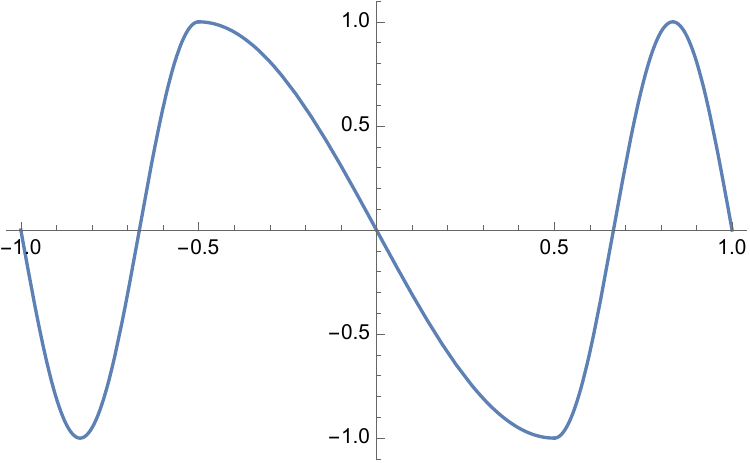}
		\end{subfigure}\hfill
		\begin{subfigure}{0.3\textwidth}
			\includegraphics[width=\textwidth]{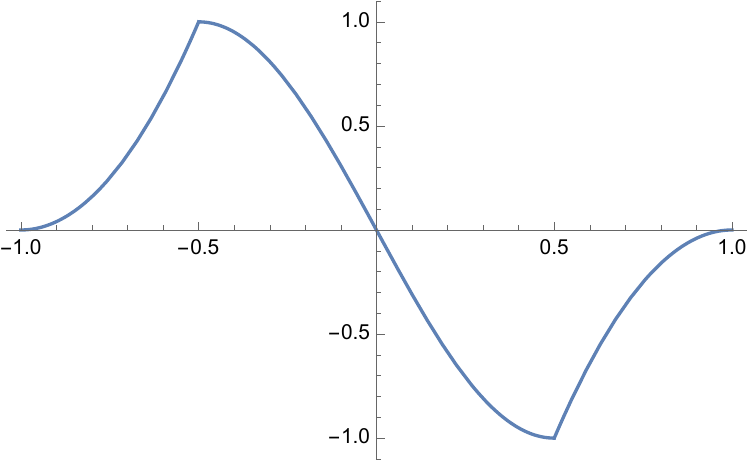}
		\end{subfigure}\hfill
		\begin{subfigure}{0.3\textwidth}
			\includegraphics[width=\textwidth]{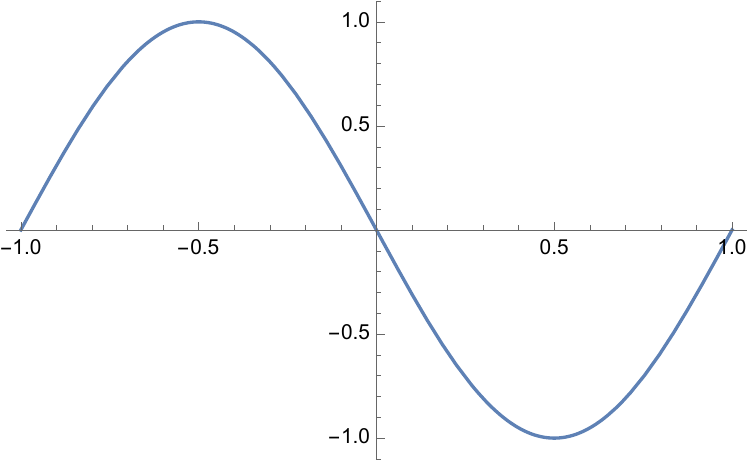}
		\end{subfigure}\vspace{1em}
		\caption{The family of flow-shape odd parts $\{V_{\rm in, o}^\delta\}_{\delta\in [0,1]}$ in the case of a unique zero in $(-H,H)$.}
		\label{odd-part-one}
	\end{figure}

We still need to ensure condition (d). To this end, we slightly modify the construction by setting
	   \begin{equation*}
        V^\delta_{\rm in}=\frac{(V_{\rm in,e} + V^\delta_{\rm in,o})\|V_{\rm in}\|_{W^{1,\infty}(-H,H)}}{\|V_{\rm in,e} + V^\delta_{\rm in,o}\|_{W^{1,\infty}(-H,H)}} \ \  \mbox{and} \ \ V^\delta_{\rm out}= \frac{(V_{\rm out,e} + V^\delta_{\rm out,o})\|V_{\rm out}\|_{W^{1,\infty}(-H,H)}}{\|V_{\rm out,e} + V^\delta_{\rm out,o}\|_{W^{1,\infty}(-H,H)}}
	   \end{equation*}
so that, for any $\delta\in[0,1]$,
$$\|V^\delta_{\rm in}\|_{W^{1,\infty}(-H,H)} + \|V^\delta_{\rm out}\|_{W^{1,\infty}(-H,H)}= \|V_{\rm in}\|_{W^{1,\infty}(-H,H)} + \|V_{\rm out}\|_{W^{1,\infty}(-H,H)} \leq r .$$
The construction is now complete within the class $\mathcal{F}_{r,0}$.\par
It follows from Proposition \ref{exi-uni-bv} that for all $(\eps,\delta)\in[0,1]^2$ there exists a unique solution $(u_{\eps,\delta}, p_{\eps,\delta})$ to \eqref{bv-pb} in $\Omega_\eps=R\setminus B_\eps$ with flow $\lambda(V^\delta_{\rm in}, V^\delta_{\rm out})$ for $\lambda\in [0, \Lambda]$, and $\Lambda>0$ is independent of both $\eps$ and $\delta$. We define the map
	\begin{equation*}\begin{aligned}
	\Phi: \quad \{ B_\eps\}_{\eps\in[0,1]} \times \{(V^\delta_{\rm in}, V^\delta_{\rm out})\}_{\delta\in[0,1]}  \quad  &\rightarrow  \qquad  \mathbb{R}\\
	\left(B_\eps, V^\delta_{\rm in}, V^\delta_{\rm out}\right)\qquad\qquad\quad &\mapsto\quad\mathcal{L}_{B_\eps}(u_{\eps,\delta},p_{\eps,\delta}).
	\end{aligned}
	\end{equation*}
By \eqref{lift-neg}  and the continuous dependence results in Theorems \ref{theo-cont-flows} and \ref{theo-cont-bodies}, the generalized Bolzano Theorem yields the existence of $(\widetilde{\eps}, \widetilde{\delta})\in (0,1)\times (0,1)$ such that
	$\Phi(B_{\widetilde{\eps}}, V^{\widetilde{\delta}}_{\rm in}, V^{\widetilde{\delta}}_{\rm out})=0$.
Equivalently, denoting $\widetilde{B}=B_{\widetilde{\eps}}$ and $( \widetilde{V}_{\rm in}, \widetilde{V}_{\rm out})=(V^{\widetilde{\delta}}_{\rm in}, V^{\widetilde{\delta}}_{\rm out})$, there exists a unique solution $(u_{\widetilde{\eps},\widetilde{\delta} }, p_{\widetilde{\eps},\widetilde{\delta} })$ to \eqref{FSI} in $\widetilde{\Omega}=R\setminus \widetilde{B}$ with flow $\lambda( \widetilde{V}_{\rm in}, \widetilde{V}_{\rm out})$.\par
Using analogous arguments, the same result holds also
when $B$ is symmetric and $(V_{\rm in}, V_{\rm out})$ is non-even and when $B$ is asymmetric and $(V_{\rm in}, V_{\rm out})$
is even.\end{proof}

\section{A new measure for the stability of bridges}\label{sec-measure}

Theorem \ref{theo-zerolift-loc} states that, for a fixed flow magnitude $\lambda\in[0,\Lambda]$, in wind tunnel experiments one may create asymmetric flow
shapes and/or asymmetric body shapes in order to ``artificially'' obtain an equilibrium configuration in which no lift acts on the body. This local result is only partially satisfactory since the target is to find a global result, namely the construction of a body shape minimizing
the lift action in a given range of flow magnitudes. In this section we address this query.\par

Consider a fixed body shape $B\in\mathcal{C}_{\alpha,D}$. From \cite{BocGaz23} we know that the lift $\mathcal{L}_B(u,p)$ in \eqref{lift-weak}
belongs to $C([0,\Lambda])$, with $\Lambda>0$ possibly smaller than the uniqueness threshold in Proposition \ref{exi-uni-bv}.
Fix $U\in\{0,1\}$ and consider the functional
\begin{equation}\label{mappa}
\begin{aligned}
\mathfrak{L}_B:\quad \quad \mathcal{F}_{r,U} \ \quad  &\quad \rightarrow \qquad \qquad  C([0,\Lambda])\\
(V_{\rm in}, V_{\rm out})  &\quad\mapsto\quad \mathfrak{L}_B[\lambda (V_{\rm in}, V_{\rm out})]= \mathcal{L}_B(u,p)
\end{aligned}
\end{equation}
which maps the flow shapes to the lift exerted on $B$, seen as a continuous function of $\lambda$.
We restrict \eqref{mappa} to flow shapes that belongs to $\mathcal{F}_{r,U}$, introduced in Definition \ref{FU}.
This set has an evident physical interpretation: if a bridge of cross-section $B$ is built in some region where the winds are controlled by
the parameter $r>0$ (not necessarily small), then $\mathcal{F}_{r,U}$ contains all the ``expected winds''.
\vspace{1em}
\begin{definition}\label{instabmeas}
Let $r>0$, $U\in\{0,1\}$ and let $D\subset R$, $0<\alpha<|D|$. We call {\bf instability measure} of the body shape $B\in\mathcal{C}_{\alpha,D}$ the positive number
\begin{equation}\label{gammaruB}
\gamma_{r,U}(B):=\sup\limits_{(V_{\rm in}, V_{\rm out})\in\mathcal{F}_{r,U}} \| \mathfrak{L}_B[ \ \cdot \ (V_{\rm in},  V_{\rm out})]\|_{C([0,\Lambda])}.
\end{equation}
\end{definition}

The main result of this paper is the existence of a body shape of given area minimizing the instability measure.
\vspace{1em}
\begin{theorem}\label{existmax}
Let $r>0$, $U\in\{0,1\}$ and let $D\subset R$, $0<\alpha<|D|$ and $\mathcal{C}_{\alpha, D}$ be as in Definition \ref{def-admbodies}.
There exists $B^* \in \mathcal{C}_{\alpha, D} $ such that
\begin{equation}\label{min-gamma}
\gamma_{r,U}(B^*)\leq \gamma_{r,U} (B) \qquad \mbox{for all} \ B\in \mathcal{C}_{\alpha, D}.
\end{equation}
\end{theorem}
\begin{proof} We first claim that the supremum in \eqref{gammaruB} is attained, namely that, for a given $B\in\mathcal{C}_{\alpha,D}$,
there exists $(\overline{V}_{\rm in},  \overline{V}_{\rm out})\in \mathcal{F}_{r,U}$ such that
\begin{equation}\label{supmax}
\| \mathfrak{L}_B[ \ \cdot \ (\overline{V}_{\rm in},  \overline{V}_{\rm out})]\|_{C([0,\Lambda])}=\max\limits_{(V_{\rm in}, V_{\rm out})\in\mathcal{F}_{r,U}} \| \mathfrak{L}_B[ \ \cdot \ (V_{\rm in},  V_{\rm out})]\|_{C([0,\Lambda])}.
\end{equation}
To this end, consider a maximizing sequence $\{(V^n_{\rm in},  V^n_{\rm out})\}_{n\in \mathbb{N}}\subset \mathcal{F}_{r,U}$, that is,
\begin{equation*}
\|\mathfrak{L}_B[ \ \cdot \ (V^n_{\rm in},  V^n_{\rm out})]\|_{C([0,\Lambda])}
\rightarrow \sup\limits_{(V_{\rm in}, V_{\rm out})\in\mathcal{F}_{r,U}} \| \mathfrak{L}_B[ \ \cdot \ (V_{\rm in},  V_{\rm out})]\|_{C([0,\Lambda])} \quad\mbox{as}\quad n\rightarrow \infty.
		\end{equation*}
Since $\{(V^n_{\rm in},  V^n_{\rm out})\}_{n\in \mathbb{N}}$ is uniformly bounded in $W^{1,\infty}(-H,H)^2$, from the Banach-Alaoglu Theorem
we know that, up to a subsequence, $(V^n_{\rm in},  V^n_{\rm out})$ weakly-$\ast$ converges to some $(\overline{V}_{\rm in},  \overline{V}_{\rm out})\in W^{1,\infty}(-H,H)^2$, which also belongs to $\mathcal{F}_{r,U}$ due to Proposition \ref{closed-flows}.
By Theorem \ref{theo-cont-flows} we then know that
\begin{equation*}
\|\mathfrak{L}_B[\ \cdot \ (V^n_{\rm in},  V^n_{\rm out})]\|_{C([0,\Lambda])}\rightarrow \| \mathfrak{L}_B[ \ \cdot \ (\overline{V}_{\rm in},  \overline{V}_{\rm out})]\|_{C([0,\Lambda])}   \quad\mbox{as}\quad n\rightarrow \infty,
\end{equation*}
and that \eqref{supmax} holds.\par
By Theorem \ref{theo-cont-bodies}, we have that the lift $\mathcal{L}_B(u,p)$ is a continuous function of $B$ in the class $\mathcal{C}_{\alpha,D}$, therefore $\gamma_{r,U}(\cdot)$ is a continuous function in $\mathcal{C}_{\alpha,D}$. By Proposition \ref{compact-bodies}, the generalized Weierstrass Theorem guarantees the existence of a body shape $B^*\in \mathcal{C}_{\alpha,D}$ that minimizes $\gamma_{r,U}$ in $\mathcal{C}_{\alpha,D}$, which proves \eqref{min-gamma}.\end{proof}

The instability measure $\gamma_{r,U}(B)$, defined in \eqref{gammaruB}, appears suitable to quantify the instability of $B$ since,
for a given range of flow magnitudes, it computes the maximum strength of the lift exerted on $B$. Theorem \ref{existmax} shows that,
within the class $\mathcal{C}_{\alpha, D}$, an optimal shape $B^*$ minimizing $\gamma_{r,U}(\cdot)$ exists. Correspondingly, this
shows that an optimal cross-section for the deck of the bridge exists (maximizing its stability) but leaves open how to find it or, at least, some qualitative features of it. In view of the results in \cite{Piro74}, a conjecture could be that the optimal profile has a front end shaped like a wedge of right angle, see
also \cite{montoya2018} for the engineering point of view. Since in real applications the wind flow is of Couette-type, that is, $U=1$, we do not expect a symmetric body shape minimizing $\gamma_{r,U}$.\par \vspace{3em}\noindent
{\bf Acknowledgements.} The authors were partially supported by the Gruppo Nazionale per l’Analisi Matematica, la Probabilità e le loro Applicazioni (GNAMPA) of the Istituto Nazionale di Alta Matematica (INdAM). The first author was supported by the EU research and innovation programme Horizon Europe through the Marie Sklodowska-Curie project THANAFSI (No. 101109475). This work is part of the PRIN project 2022 "Partial differential equations and related geometric-functional inequalities", financially supported by the EU, in the framework of the "Next Generation EU initiative".
\vspace{1em}
\section*{Declarations}

\begin{itemize}
\item Funding: This research received no specific grant from any funding agency in the public, commercial, or
not-for-profit sectors.
\item Conflict of interest: Not applicable
\item Ethics approval: This study has not been duplicate publication or submission elsewhere.
\item Consent to participate: Not applicable
\item Consent for publication: Not applicable
\item Availability of data and materials: Not applicable
\item Code availability: Not applicable
\item Authors' contributions: Not applicable
\end{itemize}

\bibliography{references}
\end{document}